\documentclass[11pt,dvipsnames]{amsart}
\usepackage[margin=1in]{geometry}
\usepackage{graphicx, amsmath, amssymb, amsthm, amsfonts}
\usepackage[mathscr]{eucal}
\usepackage{thmtools} 
\usepackage{microtype}
\usepackage{tikz-cd}
\usepackage{xparse}         
\usepackage{verbatim} 
\usepackage{xy}
\usepackage[all]{xypic}
\usepackage{spectralsequences}
\usepackage{mathtools}
\usepackage[textsize=tiny]{todonotes}
\usepackage[pagebackref, colorlinks, citecolor=Red, linkcolor=NavyBlue, urlcolor=NavyBlue]{hyperref}
\usepackage[capitalise, nameinlink]{cleveref}
\hypersetup{linktoc=all} 

\author[Chan]{David Chan}
\address[Chan]{Department of Mathematics, Michigan State University, East Lansing, MI 48824 }
\email{chandav2@msu.edu}

\author[Gotliboym]{Marc Gotliboym}
\address[Gotliboym]{Department of Mathematics, Michigan State University, East Lansing, MI 48824 }
\email{marcg57@msu.edu}

\author[Klang]{Inbar Klang}
\address[Klang]{Department of Mathematics,
Vrije Universiteit Amsterdam - Faculty of Science,
De Boelelaan 1111,
1081 HV Amsterdam,
The Netherlands}
\email{i.klang@vu.nl}

\author[Wisdom]{Noah Wisdom}
\address[Wisdom]{Department of Mathematics,
         Northwestern University,
         Evanston, IL, 60208
         USA}
\email{NoahAnkney2026@u.northwestern.edu}

\newcommand{\C}{\mathbb{C}}
\newcommand{\R}{\mathbb{R}}
\newcommand{\Z}{\mathbb{Z}}

\newcommand{\F}{\mathbb{F}}

\DeclareMathOperator{\colim}{colim}
\DeclareMathOperator{\id}{id}

\newcommand{\Alg}{\mathsf{Alg}}

\newcommand{\ETHH}{\mathrm{ETHH}}
\newcommand{\ethh}{\Res^{C_p\times S^1}_{C_p\times 1}\mathrm{ETHH}}
\newcommand{\THH}{\mathrm{THH}}

\newcommand{\TC}{\mathrm{TC}}

\newcommand{\Sp}{\mathrm{Sp}}

\newtheorem{theorem}{Theorem}[section]
\newtheorem{lemma}[theorem]{Lemma}

\newtheorem{proposition}[theorem]{Proposition}
\newtheorem{corollary}[theorem]{Corollary}

\newtheoremstyle{BoldRemark} 
{10pt}                    
{10pt}                    
{\upshape}                   
{}                           
{\bfseries}                  
{.}                          
{.5em}                       
{}  
\theoremstyle{BoldRemark}
\newtheorem{remark}[theorem]{Remark}
\newtheorem{definition}[theorem]{Definition}
\newtheorem{example}[theorem]{Example}

\newenvironment{cd}{
  \begin{center}\begin{tikzcd}}{
    \end{tikzcd}\end{center}}

\newcommand{\mac}[1]{\underline{#1}} 

\DeclareRobustCommand{\infcat}{%
  \ifmmode
   \mathrm{Cat}_\infty%
  \else
   \ensuremath{\infty}-category%
  \fi
}

\newcommand{\we}{\simeq} 

\newcommand{\smsh}{\wedge}
\newcommand{\cyc}{\mathrm{cyc}}

\newcommand{\Res}{\operatorname{Res}}

\newcommand{\Pic}{\operatorname{Pic}}

\newcommand{\Fun}{\mathrm{Fun}}
\NewDocumentCommand{\Spg}{O{G}}{\Sp_{#1}}
\NewDocumentCommand{\mspg}{O{G}}{\mac{\Sp}_{#1}}

\NewDocumentCommand{\etc}{ O{G} }{\mac{\mathrm{ETC}}}

\NewDocumentCommand{\horb}{O{G}}{_{h#1}}
\NewDocumentCommand{\hfix}{O{G}}{^{h#1}}
\NewDocumentCommand{\tate}{O{G}}{^{t#1}}
\NewDocumentCommand{\geo}{O{G}}{^{\Phi #1}}

\newcommand{\hfixp}{\hfix[C_p]}
\newcommand{\tatep}{\tate[C_p]}
\newcommand{\geop}{\geo[C_p]}

\NewDocumentCommand{\thfix}{O{G} O{N}}{^{h_{#1}#2}} 
\NewDocumentCommand{\thorb}{O{G} O{N}}{_{h_{#1}#2}} 
\NewDocumentCommand{\ttate}{O{G} O{N}}{^{t_{#1}#2}}
\NewDocumentCommand{\macttate}{O{G} O{N}}{^{\mac{t}_{#1}#2}}
\NewDocumentCommand{\macthfix}{O{G} O{N}}{^{\mac{h}_{#1}#2}}
\NewDocumentCommand{\macthorb}{O{G} O{N}}{_{\mac{h}_{#1}#2}}
\NewDocumentCommand{\macthsp}{O{G} O{N}}{{\mac\Sp}\thfix[#1][#2]}
\NewDocumentCommand{\thsp}{O{G} O{N}}{{\Sp}\thfix[#1][#2]}
\NewDocumentCommand{\hsp}{O{G}}{{\Sp}\hfix[#1]}

\newcommand{\HFp}{H\underline{\mathbb{F}}_p}

\newcommand{\Tor}{\mathrm{Tor}}
\newcommand{\Top}{\mathrm{Top}}

\NewDocumentCommand{\gLEq}{G{F} G{F'} G{\mathscr{C}} }{%
  \mac{\mathrm{LEq}}_{{#1}:{#2}}\left(#3\right)}

\theoremstyle{plain}
\newtheorem{introtheorem}{Theorem}
 
\crefname{introtheorem}{Theorem}{Theorems}

\title{Computations in equivariant topological Hochschild homology}
\date{}

\begin{document}

\begin{abstract}
    One of the most effective approaches to computations in algebraic $K$-theory is trace methods, which compare algebraic $K$-theory with topological Hochschild homology and topological cyclic homology. In recent work, two of the authors, together with Gerhardt, construct an equivariant refinement of topological Hochschild homology (ETHH) which receives a trace map from Merling's genuine equivariant algebraic $K$-theory. In this paper, we perform foundational computations of ETHH that can serve as input for future computations of ETHH and equivariant topological cyclic homology.  Namely, we compute $\mathrm{ETHH}(H\underline{\F}_p)$ for odd primes, showcasing the complexity of B\"okstedt periodicity in this setting.  Furthermore, we give computations of $\mathrm{ETHH}$ for the equivariant complex cobordism spectra $MU_G$ and $MU_{\R}$.
\end{abstract}
\maketitle
\tableofcontents

\section{Introduction}

Recent years have seen a proliferation of new ideas and techniques in stable homotopy theory, particularly in  equivariant techniques and algebraic K-theory.  At the intersection of these two areas is \emph{genuine equivariant algebraic $K$-theory}, which studies invariants of rings (and ring spectra) with an action by a compact Lie group $G$. Equivariant algebraic $K$-theory is related to special values of Artin $L$-functions, and plays a central role in Malkiewich and Merling's recent work on a refinement of the stable parametrized $h$-cobordism theorem \cite{ElmantoZhang,MM1,MM2}. It has also been effectively utilized by Vogeli to perform new, non-equivariant computations in the algebraic $K$-theory of group algebras \cite{Vogeli}. 
Unfortunately, computations in non-equivariant algebraic K-theory are difficult, and are often much more difficult in the equivariant context.

One very successful approach to computation in algebraic $K$-theory is trace methods, in which one approximates the algebraic $K$-theory of a ring $R$ using the \emph{cyclotomic trace} $K(R)\to \TC(R)$, which is often a good approximation.  Here, $\TC(R)$ denotes \emph{topological cyclic homology}, which is derived from \emph{topological Hochschild homology}, denoted $\THH(R)$, using the fact that $\THH(R)$ is a cyclotomic spectrum. The \emph{Dennis trace} gives a map $K(R)\to \THH(R)$, which plays a crucial role in the construction of the cyclotomic trace. 

A natural question to ask is whether the tools from trace methods can be developed to aid in computations of the equivariant algebraic $K$-theory of a $G$-ring $R$. In \cite{CGK25}, some of the authors took a step in this direction, introducing equivariant THH, denoted $\ETHH$, which receives an equivariant refinement of the Dennis trace, $K_G(R)^G\to\ETHH(R)^G$.  Here, $K_G(R)$ is the equivariant algebraic $K$-theory $G$-spectrum, in the sense of Merling \cite{Merling}. 

Recent work of Hilman--Ramzi \cite{HilmanRamzi} develops a notion of equivariant THH in a slightly different setting; upcoming work of the second named author will show equivalence of the definitions for $G$-ring spectra. Hilman--Ramzi also develop an equivariant Dennis trace from the equivariant algebraic $K$-theory of Barwick et al.\ \cite{Bar17, BGS20} to their version of equivariant THH. The relationship between the Dennis trace maps of \cite{HilmanRamzi} and of \cite{CGK25} is unclear.

Equivariant $\THH$ comes equipped with a circle action, and forthcoming work of the second named author will use the circle action to describe the sense in which $\ETHH(R)$ is cyclotomic and use this structure to define an equivariant notion of $\TC$ \cite{Got}.

While we now have an equivariant refinement of the Dennis trace, the effectiveness of a trace methods approach depends heavily on the computability of the homotopy groups of $\ETHH(R)$.  The goal of the present paper is to provide computations of homotopy groups $\ETHH(R)$ for several fundamental ring $G$-spectra.

While we now have an equivariant refinement of the Dennis trace, the effectiveness of a trace methods approach depends heavily on the availability of fundamental computations of both $\ETHH(R)$ and of equivariant $K$-theory from which to build upon.  While some computations have been carried out,  for instance \cite{AGHKK22,ChanVogeli,CW25,Wis26}, the literature of explicit computations remains somewhat sparse. With this in mind, the goal of the present paper is to provide computations of the homotopy groups of $\ETHH(R)$ for several fundamental ring $G$-spectra.

\subsection{Equivariant B\"okstedt Periodicity}

The first computations of THH were done by B\"okstedt, who computed $\THH(H\mathbb{Z})$ and $\THH(H\F_p)$ for all primes $p$. These computations play a fundamental role in the theory, and continue to be important building blocks for new computations \cite{KrauseNikolaus:BokstedtPeriodicity,BhattMorrowScholze}.  In the $C_p$-equivariant setting, the correct analogue of $\mathbb{F}_p$ is the \emph{constant Green functor} $\underline{\F}_p$, and our first main theorem computes the homotopy groups of the $C_p$-geometric fixed points of $\ETHH(\HFp)$.

We write $\Gamma$ and $\Lambda$ for divided power and exterior algebras over $\F _p$, respectively.

\begin{introtheorem}[{\cref{lemma: injection from gen to geo,thm:Bokstedt}}]\label{thm:introBokstedt}
    Let $p$ be an odd prime. The graded homotopy ring $\pi^{C_p}_*(\ETHH(\HFp))$ is isomorphic to the subring of $\Gamma[a]\otimes\F_p[b,c]\otimes \Lambda[d,e]$ generated by $b$ and all monomials divisible by either $e$ or a divided power of $a$. The degrees are $|a|=|b|=|c|=2$, $|d|=1$, and $|e|=3$.
\end{introtheorem} 

In comparison, B\"okstedt's foundational result showed that $\pi_*\THH(H\F _p) \cong \F _p [b]$, where $b$ is a polynomial generator in degree 2. This result is known as B\"okstedt periodicity. Our theorem demonstrates that equivariant B\"okstedt periodicity is far more complicated. 

We prove \cref{thm:introBokstedt} using a careful analysis of the Tate square,
\[
    \begin{tikzcd}
        \ETHH(\HFp)^{C_p}\ar[r,"i"] \ar[d] & \ETHH(\HFp)\geop \ar[d]\\
        \ETHH(\HFp)^{hC_p} \ar[r] & \ETHH(\HFp)^{tC_p},
    \end{tikzcd}
\]
which is a pullback of commutative ring spectra. The geometric fixed points are interesting in their own right.

\begin{introtheorem}[{\cref{prop-THH-geomfp,cor: homotopy of geo of ETHH hfp}}]\label{introthm: geo of ETHH hfp}
    Let $p$ be an odd prime.  There is an equivalence of $\mathbb{E}_1$-$H\F_p$-algebras 
    \[
        (\ethh(\HFp))\geop \simeq H\F_p \wedge (\Omega S^3)_+ \wedge S^1_+ \wedge \C P^\infty _+ \wedge (LS^3)_+ 
    \]
    where $LS^3$ denotes the free loop space. Consequently, there is an isomorphism of graded rings
    \[
        \pi_*\ETHH(\HFp)\geop\cong \Gamma[a]\otimes\F_p[b,c]\otimes \Lambda[d,e]
    \]
    where $a$, $b$, $c$, $d$, and $e$ have the same degrees as in \cref{thm:introBokstedt}.
\end{introtheorem}

The proof of \cref{introthm: geo of ETHH hfp} utilizes the fact that $\ETHH$ commutes with the geometric fixed points construction (see \cref{cor-geomfp-ETHH}). In particular, this reduces the computation to the non-equivariant computation of $\THH(\HFp\geop)$. From here, we prove an equivalence of $\mathbb{E}_2$-$H\F_p$-algebras $\HFp\geop\simeq H\F_p\wedge S^1_+\wedge \Omega S^3_+$ and use the monoidal properties of THH to prove the theorem.

Returning to the proof of \cref{thm:introBokstedt}, we see that on homotopy, the fixed points are a subring of the geometric fixed points. To fully pin down the subring, the final, key part of the analysis is understanding how the right vertical map acts on the generators of the homotopy of $\ETHH(\HFp)\geop$.

When $p=2$, the $C_2$-homotopy groups of $\ETHH(H\underline{\F}_2)$ are computed in \cite{AGHKK22} using the fact that $H\underline{\F}_2$ is a sufficiently multiplicative Thom spectrum \cite{BW18,HW20}. The computation of the ring structures on the homotopy groups of $\ETHH(H\underline{\F}_2)\geo[C_2]$ and $\ETHH(H\underline{\F}_2)^{C_2}$ is more difficult and requires a different method than the one we apply here in the odd prime case.  This computation will appear in forthcoming work of the first-named author and Chase Vogeli. 

As a future direction, we think it would be interesting to understand the $RO(C_p)$-graded homotopy ring of $\ETHH(\HFp)$.  Oftentimes, this ring has somewhat better formal properties than the integer graded homotopy ring, and could shed some light on the somewhat unorthodox answer obtained in \cref{thm:introBokstedt}. Additionally, an understanding of the $RO(C_p)$-graded homotopy groups would give access to norm structures on homotopy groups which could prove useful for later computation related to equivariant topological cyclic homology.

\subsection{Equivariant complex cobordism}

One of the most important ring spectra is the complex cobordism spectrum $MU$, which plays a fundamental role in computations of stable homotopy groups. The homotopy $S^1$-fixed points of $\THH$ of some quotients of $MU$ have been used to resolve the telescope conjecture \cite{BHLS23} and to study redshift \cite{HW22}. A full computation of $\TC(MU)$ could be used to compute information about $\TC$ of the sphere spectrum, an invariant with applications to geometric topology \cite{Hesselholt:Whitehead,Rognes:Cobordism}. The equivariant refinement of the complex cobordism spectrum, denoted $MU_G$, has been studied extensively; see, for instance, \cite{GreenleesMay:Localization,Hausmann}. Note that we are working with tom Dieck's homotopical equivariant cobordism (a Thom spectrum carrying the universal $G$-equivariant complex orientation \cite{CGK02}) as opposed to geometric cobordism (which encodes information about cobordism classes of $G$-manifolds).

The non-equivariant computation of $\THH(MU)$ can be done using the fact that $MU$ is a Thom spectrum.  In particular, \cite{BCS10} gives an equivalence
\[
    \THH(MU)\simeq MU\wedge SU_+
\]
where $SU$ denotes the infinite special unitary group. From here, the homotopy groups can be computed using the Atiyah--Hirzebruch spectral sequence.

In this paper, we upgrade the compatibility of equivariant factorization homology with Thom spectra proven in \cite{HHKWZ24} to include stronger multiplicative properties. We use this to compute $\ETHH(MU_G)$ as a $G$-$\mathbb{E}_\infty$ ring spectrum.
\begin{introtheorem}[{\cref{cor:identification-of-ETHH-of-MU_G}}]\label{introthm:identification-of-ETHH-of-MU_G}
    Let $G$ be a finite group. There is an equivalence
    \[
        \mathrm{Res}_{G}^{G \times S^1} \mathrm{ETHH}(MU_G) \simeq MU_G \wedge \Sigma^\infty _+ SU_G
    \]
    of $G$-$\mathbb{E}_\infty$ ring spectra, where $SU_G$ denotes an appropriately equivariant infinite special unitary group.
\end{introtheorem}

When $G$ is abelian, we use this to give a computation of the $G$-equivariant homotopy groups of $\mathrm{ETHH}(MU_G)$. 

\begin{introtheorem}[{\cref{thm:MU_G-homology-of-SU_G-for-G-abelian}}]\label{introthm: homotopy of ETHH(MU_G)}
    Let $G$ be compact, abelian Lie group and $H \subset G$. Then we have
    \[
        \pi^H_*(MU_G \wedge (SU_G)_+) \cong (MU_H)_* \otimes_\Z \Lambda(\lambda_1,\lambda_2,...)
    \]
    where $\lambda_i$ is an exterior algebra generator in degree $2i+1$. In light of \cref{introthm:identification-of-ETHH-of-MU_G}, this computes $\pi^H_*(\ETHH(MU_G))$ when $G$ is finite abelian.
\end{introtheorem}

The proof of \cref{introthm: homotopy of ETHH(MU_G)} relies on an explicit $G$-CW structure for $SU_G$ which we construct in \cref{section: MUG}.  This cell structure is a $G$-equivariant refinement of a CW-structure on $SU$ due to Yokota \cite{Yok56}.

When $G=C_2$, we also consider the Real analogue of complex cobordism, denoted $MU_{\R}$.  This spectrum is crucial to geometric applications of equivariant homotopy. For instance, it plays a central role in Hill, Hopkins, and Ravenel's work on the Kervaire invariant one problem \cite{HHR}. We are able to give a complete description of $\ETHH(MU_{\R})$, as well as its $RO(C_2)$-graded homotopy ring.

\begin{introtheorem}[{\cref{cor-ETHH-MUR,cor: homotopy of ETHHMUR}}]\label{introthm-ETHH-MUR}
    There is an equivalence of $C_2$-$\mathbb{E}_\infty$ ring spectra
    \[
        \mathrm{Res}_{C_2}^{C_2 \times S^1} \mathrm{ETHH}(MU_\mathbb{R}) \simeq MU_\mathbb{R} \wedge \Sigma^\infty _+ B(BU_\mathbb{R})
    \]
    where $BU_{\mathbb{R}}$ denotes the classifying space of stable Real vector bundles.  Consequently, there is an isomorphism 
    \[ 
        \underline{\pi}_\bigstar \Res_{C_2}^{C_2 \times S^1} \ETHH(MU_\R) \cong \Lambda_{(\underline{MU_\R})_\bigstar} (\overline{\lambda}_n | n \geq 1)
    \]
    of $RO(C_2)$-graded Green functors, where $\overline{\lambda}_n$ has degree $n \rho+1$.
\end{introtheorem}

This computation uses the multiplicative $G$-equivariant bar spectral sequence (\cref{thm-ebar,thm-ebar-multiplicative}), which may be of independent interest.

\subsection{Outline}

The remainder of the paper is organized as follows. Background material on ETHH, equivariant Thom spectra, and factorization homology can be found \cref{sec: prelims}.  In \cref{sec:barss}, we construct a multiplicative $G$-equivariant bar spectral sequence which we need for computing $\ETHH(MU_\R)$.  The proofs of \cref{introthm-ETHH-MUR,introthm:identification-of-ETHH-of-MU_G,introthm: homotopy of ETHH(MU_G)} are in \cref{sec: ETHH of Thom}.  Finally, we prove \cref{introthm: geo of ETHH hfp,thm:introBokstedt} in \cref{sec: Bokstedt}.

\subsection{Acknowledgments}
We would like to thank Ben Antieau, Teena Gerhardt, Mike Hill, Liam Keenan, Tyler Lawson, Maximilien P\'eroux, J.D. Quigley, and Chase Vogeli for helpful conversations related to this work.  DC and MG were partially supported by NSF grant DMS-2135960. IK is grateful for funding from the NWO-XL grant ``Symmetry on the interface of topology and higher algebra", which supported her travel for collaboration on this project. 
\section{Preliminaries}\label{sec: prelims}
In this section we recall some necessary background material.  
\subsection{\texorpdfstring{Review of $\mathrm{ETHH}$}{Review of ETHH}}

Let $G$ be a finite group. We write $\Spg$ for the category of genuine orthogonal $G$-spectra, and $\Alg(\Spg)$ for the category of associative ring $G$-spectra. We recall the definition of $\mathrm{ETHH}$ from \cite{CGK25}.

\begin{definition}
  Let $R$ be an associative algebra in $\Spg$. The \textit{cyclic bar construction} of $R$, $N^{\cyc}_\bullet(R)$, is a simplicial object given in degree $n$ by $R^{\smsh n+1}$ with face and degeneracy maps as follows:
  \begin{alignat*}{3}
    d_0 &= (\id^{\smsh n-1} \smsh \mu) \circ \tau  \\
    d_i &= (\id^{\smsh i-1}\smsh \mu \smsh \id^{\smsh n-i}) \qquad& 0&< i<n \\
    s_i &= \id^{\smsh i+1} \smsh \eta \smsh \id^{\smsh n-i} \qquad& 0&\leq i<n, 
  \end{alignat*}
  where $\mu$ and $\eta$ are the multiplication and unit in $R$, respectively, and $\tau$ rotates the first factor of $R$ to the end.
\end{definition}

\begin{definition}[{\cite[Definition 3.1]{CGK25}}]
  Let $\mathcal{U}$ be a complete $(G\times S^1)$-universe, and let $\mathcal{V} = i^*_G\mathcal{U}$ be its restriction to a complete $G$-universe.
  Equivariant topological Hochschild homology is the functor $\ETHH\colon\Alg(\Spg) \to \Spg[G\times S^1]$ defined as
  \[ \ETHH(-) = \mathcal{I}^{\mathcal{U}}_{\mathcal{V}}|N^\cyc_\bullet(-)| \]
  where $\mathcal{I}^{\mathcal{U}}_{\mathcal{V}}$ is the change of universe functor (\cite[Section II.1]{LMSM}) and $| \cdot |$ denotes geometric realization.
\end{definition}

\begin{proposition}
    \label{ETHHisCyclicBar}
  Let $R$ be a ring $G$-spectrum. After restricting to a $G$-spectrum, $\ETHH(R)$ agrees with the cyclic bar construction in the category of $G$-spectra:
  \[ \Res^{G\times S^1}_{G\times 1}(R) \we |N^\cyc_\bullet(R)|. \]
\end{proposition}
\begin{proof}
  This follows as restriction is (strong) symmetric monoidal and preserves geometric realization.
\end{proof}

\begin{proposition}
  Equivariant topological Hochschild homology preserves commutative algebras, that is, $G$-$\mathbb{E}_\infty$-algebras.
\end{proposition}
\begin{proof}
  This is \cite[Proposition 3.5]{CGK25}.
\end{proof}

\subsection{Review of equivariant Thom constructions}
To compute $\ETHH$ of the equivariant cobordism spectra $MU_\R$ and $MU_G$ (for $G$ a finite abelian group), we will use the fact that these are equivariant Thom spectra of $G$-$\mathbb{E}_\infty$ maps, and that $\ETHH$ commutes with the equivariant Thom spectrum functor. A point-set model of the equivariant Thom spectra functor was constructed by Lewis and May in X.3.1 of \cite{LMSM}. For our purposes, we will need the equivariant Thom spectrum functor to be $G$-symmetric monoidal, so we will use the model constructed in \cite{HHKWZ24}.

\medskip

We denote the ($G$-category of) $G$-spaces by $\Top _G$, the ($G$-category of) genuine $G$-spectra by $\Spg$, and by $\Pic(\Spg)$ the Picard $G$-space of invertible $G$-spectra. The equivariant Thom spectrum functor takes a map of $G$-spaces $f: X \to \Pic(\Spg)$ and produces a genuine $G$-spectrum $\mathbf{Th}(f)$. We recall:

\begin{theorem}\label{thm-eThom} \cite{HHKWZ24}
    Let $G$ be a finite group. The equivariant Thom spectrum functor 
    $$\mathbf{Th}: \Top _G / \Pic(\Spg) \to \Spg$$ 
    is $G$-symmetric monoidal and commutes with colimits.
\end{theorem}

For us, the main use of this theorem will be in the following examples.

\begin{example}\label{ex-MU}
$\, $

    \begin{enumerate}
        \item Let $G = C_2$. The Real complex cobordism spectrum $MU_\R$ is the equivariant Thom spectrum of the Real J-homomorphism $BU_\R \to \Pic(\Spg)$, which is a $G$-$\mathbb{E}_\infty$ map. (See, for example, Remark 13 of \cite{HL}.)
        \item Let $G$ be any finite group. The equivariant complex cobordism spectrum $MU_G$ is the equivariant Thom spectrum of the equivariant J-homomorphism $BU_G \to \Pic(\Spg)$, which is a $G$-$\mathbb{E}_\infty$ map. (See, for example, Example 6.1.53 of \cite{Sch}.)
    \end{enumerate}

\end{example}

The theorem implies that both $MU_{\R}$ and $MU_G$ have the structure of $G$-commutative monoids in $\Sp_G$. Equivalently, they can be modeled by $G$-$\mathbb{E}_{\infty}$-ring spectra.

\subsection{Review of equivariant factorization homology}
In this subsection, we review the basic properties of equivariant factorization homology that we will need in order to compute $\ETHH(MU_\R)$ and $\ETHH(MU_G)$. 

\medskip

Factorization homology is a homology theory of manifolds, taking in an $n$-manifold $M$ equipped with a trivialization of its tangent bundle and an $\mathbb{E}_n$-algebra $A$ in a category $\mathcal{C}$ and producing $\int_M A \in \mathcal{C}$. Factorization homology is symmetric monoidal (in $M$ under disjoint union, and in $A$ under the symmetric monoidal structure in $\mathcal{C}$), functorial (in open embeddings of $n$-manifolds, and in $\mathbb{E}_n$-algebra maps), and satisfies $\otimes$-excision for collar gluings $M = M' \cup_{M_0 \times \R} M''$. Factorization homology can be defined as a left Kan extension.

\begin{example}
    If $R$ is a ring spectrum, then $\int_{S^1} R \simeq \THH (R)$.
\end{example}

In \cite{Hor}, Horev produces a genuine equivariant version of factorization homology, and shows that various equivariant versions of $\THH$ can be studied using genuine equivariant factorization homology. In this paper, we will show that the restriction of $\ETHH$ to a $G$-spectrum is an instance of equivariant factorization homology, and use this to compute the $\ETHH$ of equivariant complex cobordism spectra. The main input we will need is Theorem 3 of \cite{HHKWZ24}:

\begin{theorem}\label{thm-FH-Thom-HKZ}\cite{HHKWZ24}
    Equivariant Thom spectra commute with equivariant factorization homology. That is, if $f: \Omega ^V X \to \Pic(\Spg)$ is a map of $V$-fold loop spaces (for $V$ a finite-dimensional $G$-representation), then 
    $$\int_M \mathbf{Th}(f) \simeq \mathbf{Th}\left( \int_M \Omega^V X \to \int_M \Pic(\Spg) \to \Pic(\Spg)\right)$$
    for every $G$-manifold $M$ equipped with an equivariant isomorphism of bundles $TM \cong M \times V$.
\end{theorem}

Along with Theorem 4 of \cite{HHKWZ24}, which identifies $\int_M \Omega^V X \simeq \mathrm{Map}(M^+, X)$, we will obtain a multiplicative description of $\int_M$ of Thom spectra of $G$-$\mathbb{E}_\infty$ maps, and apply this to $\ETHH$.

\section{An equivariant bar spectral sequence}\label{sec:barss}

In order to compute $\ETHH(MU_\R)$, we will need an equivariant bar spectral sequence. For a genuine $G$-spectrum $E$, let $\underline{E}_\bigstar = \underline{\pi}_\bigstar(E)$ denote its $RO(G)$-graded homotopy Mackey functors, and for a $G$-space $X$, denote $\underline{E}_\bigstar (X) = \underline{\pi}_\bigstar (E \wedge \Sigma^\infty _+ X)$. The theorem below agrees with the Hyper-Tor spectral sequence of Lewis--Mandell (page 2 of \cite{LM}), or the bar spectral sequence of \cite{Hill:Freeness} for $G= C_2$, but we reproduce it here to endow it with additional algebraic structure.

\begin{theorem}\label{thm-ebar} \cite{LM}
    Let $G$ be a finite group, let $X$ be a well-pointed topological $G$-monoid, and let $E$ be a commutative ring $G$-spectrum such that $\underline{E}_\bigstar (X)$ is flat over $\underline{E}_\bigstar$. Then there is a strongly convergent spectral sequence with $E^2$-page
    \[
    \Tor_s ^{\underline{E}_\bigstar(X)}(\underline{E}_\bigstar, \underline{E}_\bigstar) \Rightarrow \underline{E}_{s + \bigstar}(BX).
    \]
\end{theorem}

\begin{proof}
    Let $B_\bullet X$ be the simplicial space giving the bar construction on $X$, so that $B_s X = X^s$. Since $\underline{E}_\bigstar (X)$ is flat over $\underline{E}_\bigstar$, we have a K\"unneth isomorphism, so that
    $$\underline{E}_\bigstar(X^s) \cong \underline{E}_\bigstar (X) ^{\Box_{\underline{E}_\bigstar}s}$$
    where the right-hand side is the $s$-fold box product over $E_\bigstar$. 
    
    In other words, $\underline{E}_\bigstar (B_\bullet X) \cong B^{\underline{E}_\bigstar} _\bullet(\underline{E}_\bigstar (X))$
    as simplicial Green functors.

    The skeletal filtration of $BX = |B_{\bullet}X|$, with $F_s / F_{s-1} \cong \Sigma^s (X^s)_+$, gives rise to a strongly convergent spectral sequence
    $$E^1 _{s, \bigstar} = \underline{E}_{s + \bigstar}(\Sigma^s (X^s)_+) \Rightarrow \underline{E}_{s + \bigstar} (BX)$$
    by standard arguments (see, for instance \cite[X.2.9]{EKMM}). Then
    $$E^1 _{s, \bigstar} \cong \underline{E}_\bigstar(X^s) \cong \underline{E}_\bigstar (X) ^{\Box_{\underline{E}_\bigstar} s}$$
    and as in the non-equivariant bar spectral sequence, one checks that the $d_1$ differential is identified under this isomorphism with the differential of the chain complex computing $\Tor_s ^{\underline{E}_\bigstar(X)}(\underline{E}_\bigstar, \underline{E}_\bigstar)$. Therefore, we can identify the $E^2$ page of this spectral sequence with $\Tor_s ^{\underline{E}_\bigstar(X)}(\underline{E}_\bigstar, \underline{E}_\bigstar)$, as required. 
\end{proof}

When $X$ is an $\mathbb{E}_2$-space, $BX$ is an $\mathbb{E}_1$-space, and therefore $\underline{E}_\bigstar (BX)$ is an $RO(G)$-graded Green functor. In this case, we endow the equivariant bar spectral sequence with a multiplicative structure.

\begin{theorem}\label{thm-ebar-multiplicative}
    Suppose in addition to the assumptions of Theorem \ref{thm-ebar} that $X$ is an $\mathbb{E}_2$ $G$-space. Then the equivariant bar spectral sequence is a spectral sequence of $RO(G)$-graded Green functors.
\end{theorem}

\begin{proof}
    If $X$ is an $\mathbb{E}_2$ $G$-space, then the multiplication on $X$ is an $\mathbb{E}_1$-map, and therefore $B_\bullet X$ is a simplicial object in $\mathbb{E}_1$ $G$-spaces, with the multiplication given levelwise by the multiplication on $X$. By the K\"unneth isomorphism, $B_\bullet (\underline{E}_\bigstar (X))$ is a simplicial object in $RO(G)$-graded Green functors. Upon taking the normalized chain complex, we obtain the shuffle product on the bar complex. This is the $E_1$-page of the spectral sequence. Since the simplicial filtration of $BX$ respects its multiplicative structure, all the differentials in the spectral sequence respect the multiplicative structure as well, so this is a spectral sequence of $RO(G)$-graded Green functors.
    \end{proof}

The bar spectral sequence is closely related to the B\"{o}kstedt spectral sequence. Another approach to $\ETHH(MU_\mathbb{R})$ would have been to develop a B\"okstedt spectral sequence for $\ETHH$, whose $E_2$-page is given by the equivariant Hochschild homology of Mehrle--Quigley--Stahlhauer \cite{MQS24,MQS25}.

\section{\texorpdfstring{$\mathrm{ETHH}$ of Thom spectra}{ETHH of Thom spectra}}\label{sec: ETHH of Thom}

In this section, we will view $\mathrm{ETHH}$ through the lens of genuine equivariant factorization homology \cite{Hor}. This will allow us to make conclusions about the multiplicative structure of $\mathrm{ETHH}$ of certain Thom spectra.

\begin{proposition}\label{prop-ETHH-FH}
    Let $A$ be a ring $G$-spectrum which is cofibrant as a $G$-spectrum, and let $S^1$ denote the circle on which $G$ acts trivially. Then
    $$\mathrm{Res}_{G}^{G \times S^1} \mathrm{ETHH}(A) \simeq  \int_{S^1} A.$$
\end{proposition}

\begin{proof}
    By the excision property of equivariant factorization homology, $\int_{S^1} A \simeq B(A, A \wedge A^{op}, A)$ (see, e.g., the proof of Proposition 7.2.2 of \cite{Hor}.) This is also equivalent to the cyclic bar construction, $\mathrm{Res}_{G}^{G \times S^1} \mathrm{ETHH}(A)$ (\cref{ETHHisCyclicBar}).
\end{proof}

The following is a multiplicative version of a theorem from \cite{HHKWZ24}. Let $V, W, U$ be finite-dimensional $G$-representations.

\begin{theorem}\label{FH-Thom-multiplicative}
    Let $A$ be the $G$-Thom spectrum of an $\mathbb{E}_{V \oplus W \oplus U}$-map,
  $$\Omega^{V \oplus W \oplus U} f: \Omega^{V \oplus W \oplus U}X \to \Pic(\Spg),$$
  with $\pi_k( \Omega^U X^H) = 0$ for all subgroups $H \subset G$ and $k < \dim ((V \oplus W)^H)$.
  Let $M$ be a smooth $G$-manifold of the same dimension as $V$. Suppose that $M \times W$ embeds equivariantly in $V \times W$, and that there is an equivariant embedding from the unit disk $D(V) \hookrightarrow M$ (call its image $D$). Then there is an equivalence of $\mathbb{E}_U$-algebras

$$\int_{M \times W \times U} A \simeq A \wedge \Sigma^\infty_+ \mathrm{Map}_*(M^+ - D, \Omega^{W \oplus U} X).$$
\end{theorem}

The proof below is adapted from the corresponding proof in \cite{HHKWZ24}, while keeping track of $\mathbb{E}_U$-algebra structures.

\begin{proof}
    Denote the equivariant embedding by $emb: M \times W \hookrightarrow V \times W$.
  Let $M \times W$ be $(V \oplus W)$-framed as a submanifold of $V \times W$, and similarly let $M \times W \times U$ be framed as a submanifold of $V \times W \times U$. Consider the following commutative diagram, where the top horizontal map is an equivalence by Theorem 4.0.1 of \cite{HHKWZ24}.
$$\xymatrix{
  \int_{M \times W \times U} \Omega^{V \oplus W \oplus U} X \ar[r]^-{\sim} \ar[d]^-{(\Omega^{V \oplus W \oplus U} f )_*} &  \mathrm{Map}_*(\Sigma^W (M^+), \Omega ^U X) \ar[d]^-{\Omega ^U f_*} \\
  \int_{M \times W \times U} \Pic(\Spg) \ar[r] \ar[d]^-{emb_*} & \mathrm{Map}_*(\Sigma^W(M^+), B^{V \oplus W} \Pic(\Spg)) \ar[d]^-{emb_*} \\
  \int_{V \times W \times U} \Pic(\Spg) \ar[r] \ar[d]^-\sim & \mathrm{Map}_*(S^{V \oplus W}, B^{V \oplus W} \Pic(\Spg)) \ar[d]^-\sim \\
  \Pic(\Spg) \ar[r]^-= & \Pic(\Spg)
}$$

Since the $G$-Thom spectrum is $G$-symmetric monoidal and respects factorization homology, the $G$-Thom spectrum of the left hand vertical composite is equivalent to $\int_{M \times W \times U} A $ as an $\mathbb{E}_{W \oplus U}$-algebra. The $\mathbb{E}_{W \oplus U}$-algebra structure on $\int_{M \times W \times U}$ comes from embeddings of $\coprod W \times U$ into $W \times U$.

\medskip

Note also that the non-abelian Poincar\'e duality equivalence (Theorem 4.0.1) $\int_M \Omega ^V X \simeq \mathrm{Map}_*(M^+, X)$ is natural in open embeddings in the manifold variable, and therefore the top horizontal equivalence
$$\int_{M \times W \times U} \Omega^{V \oplus W \oplus U} X \simeq \mathrm{Map}_*( \Sigma ^{V \oplus W \oplus U} M^+, X) \simeq \mathrm{Map}_*(\Sigma^W (M^+), \Omega ^U X)$$
is an equivalence of $\mathbb{E}_{W \oplus U}$-algebras. Thus $\int_{M \times W \times U} A $ is equivalent to the Thom spectrum of the right hand vertical composite as an $\mathbb{E}_{W \oplus U}$-algebra. Note that this composite is also equal to
$$\xymatrix{
\mathrm{Map}_*(\Sigma^W(M^+),  \Omega^U X) \ar[r]^-{emb_*} & \mathrm{Map}_*(S^{V \oplus W}, \Omega^U X) \ar[r]^-{\Omega^U f_*} & \mathrm{Map}_*(S^{V \oplus W}, B^{V \oplus W} \Pic(\Spg)) \simeq \Pic(\Spg).
}$$

The map $emb_*$ above is induced by the embedding $emb: M \times W \hookrightarrow V \times W$, equivalently by the Pontryagin-Thom collapse map associated to it, $S^{V \oplus W} \to \Sigma^W(M^+)$. This is a map of $U$-fold loop spaces, and therefore of $\mathbb{E}_U$-algebras. We have an inclusion of a small disk $D \cong V$ in $M$, and the cofiber sequence
$$\Sigma^W(M^+ -D) \overset{\Sigma^W i}{\longrightarrow}  \Sigma^W(M^+) \longrightarrow \Sigma^W S^V \cong S^{V \oplus W}$$
is split (up to homotopy) by this Pontryagin-Thom collapse map, as the composite collapse $(V \times W)^+ \to (M \times W)^+ \to (D \times W)^+ \cong (V \times W)^+$ is homotopic to the identity. Thus we have an equivalence of $U$-fold loop spaces
$$\xymatrix{
  (emb_*, i^*): \mathrm{Map}_*(\Sigma^W(M^+), \Omega ^U X) \ar[r]^-{\sim}
  & \mathrm{Map}_*(S^{V \oplus W}, \Omega ^U X) \times \mathrm{Map}_*(\Sigma^W(M^+ -D), \Omega ^U X) \ar[d]^-\sim \\
  & \mathrm{Map}_*(S^{V \oplus W}, \Omega ^U X) \times \mathrm{Map}_*(M^+ - D, \Omega^{W \oplus U} X).
}$$
Furthermore, this equivalence fits in the following commutative diagram of $\mathbb{E}_U$-algebras:
$$\xymatrix{
\mathrm{Map}_*(\Sigma^W(M^+), \Omega^U X) \ar[r]^-{(emb_*, i^*)} \ar[r]_-{\sim} \ar[d]^{emb_{*}} & \mathrm{Map}_*(S^{V \oplus W}, \Omega^U X) \times \mathrm{Map}_*(M^+ - D, \Omega^{W \oplus U} X) \ar[d]^{pr_1} \\
\mathrm{Map}_*(S^{V \oplus W}, \Omega ^U X) \ar[r]^-= \ar[d]^{\Omega^U f_*} & \mathrm{Map}_*(S^{V \oplus W}, \Omega^U X) \ar[d]^-{\Omega^U f_*} \\
\mathrm{Map}_*(S^{V \oplus W}, B^{V \oplus W} \Pic(\Spg)) \ar[r]^-= & \mathrm{Map}_*(S^{V \oplus W}, B^{V \oplus W} \Pic(\Spg)).
}$$
We have shown that $\displaystyle\int_{M \times W \times U}A$ is equivalent as an $\mathbb{E}_U$-algebra to the $G$-Thom spectrum of the left hand vertical composite, thus it is also equivalent as an $\mathbb{E}_U$-algebra to the $G$-Thom spectrum of the right hand vertical composite. This is equivalent as an $\mathbb{E}_U$-algebra to the $G$-Thom spectrum of $\Omega^{V \oplus W \oplus U} f$ smashed with the Thom spectrum of the null map $\mathrm{Map}_*(M^+ - D, \Omega^{W \oplus U} X) \to \Pic(\Spg)$. Thus it is equivalent as an $\mathbb{E}_U$-algebra to $A \wedge \Sigma^\infty_+ \mathrm{Map}_*(M^+ - D, \Omega^{W \oplus U} X)$, and we can conclude.
\end{proof}

We can now deduce that $\ETHH$ of the $G$-Thom spectrum of a $G$-$\mathbb{E}_\infty$ map is $G$-$\mathbb{E}_\infty$. Furthermore, 

\begin{corollary}\label{cor-ETHH-FH-mult}
    Suppose that $A = \mathbf{Th}(\Omega f)$ for an equivariant  infinite loop map $\Omega f: \Omega X \to \Pic(\Spg)$, where $X$ is $G$-connected. Then
    $$\mathrm{Res}_{G}^{G \times S^1} \mathrm{ETHH}(A) \simeq A \wedge \Sigma^\infty _+ X$$
    as $G$-$\mathbb{E}_\infty$ ring spectra.
\end{corollary}

\begin{proof}
    By Proposition \ref{prop-ETHH-FH}, $\mathrm{Res}_{G}^{G \times S^1} \mathrm{ETHH}(A) \simeq \int_{S^1} A$. Note that, since $A$ is $G$-$\mathbb{E}_\infty$, $\int_{S^1} A \simeq \int_{S^1 \times W} A $ for any representation $W$. This is because $\int_{- \times W} A$ and $\int_{-} A$ both satisfy the axioms of genuine equivariant factorization homology of \cite{Hor}, and therefore are equivalent. In particular, embed $S^1 \times \mathbb{R}$ in $\mathbb{R}^2$; then by Theorem \ref{FH-Thom-multiplicative}, for any representation $U$,
    $$\int_{S^1 \times \mathbb{R} \times U} A \simeq A \wedge \Sigma^\infty _ + \mathrm{Map}_*(S^1 _+ - D, \Omega^{ \mathbb{R} \oplus U} B^{\mathbb{R} \oplus U} X) \simeq A \wedge \Sigma^\infty _ + X$$
    as $\mathbb{E}_U$-algebras. Therefore, for any representation $U$, $\int_{S^1} A \simeq A \wedge \Sigma^\infty_+ X$ as $\mathbb{E}_U$-algebras, as required.
\end{proof}

We now turn to some computations.

\subsection{\texorpdfstring{The computation of $\ETHH(MU_\mathbb{R})$}{The computation of ETHH(MU\_R)}}
In this subsection we consider $\ETHH(MU_\mathbb{R})$, where $MU_{\mathbb{R}}$ denotes the Real bordism spectrum from \cref{ex-MU}(1). 
\begin{theorem}\label{cor-ETHH-MUR}
    There is an equivalence of $C_2$-$\mathbb{E}_\infty$ ring spectra
    \[
        \mathrm{Res}_{C_2}^{C_2 \times S^1} \mathrm{ETHH}(MU_\mathbb{R}) \simeq MU_\mathbb{R} \wedge \Sigma^\infty _+ B(BU_\mathbb{R})
    \]
\end{theorem}

\begin{proof}
    This follows from Corollary \ref{cor-ETHH-FH-mult}, along with the fact that $MU_\mathbb{R}$ is the $G$-Thom spectrum of a $G$-$\mathbb{E}_\infty$ map $BU_\mathbb{R} \to \Pic(\Spg)$.
\end{proof}

\begin{corollary}\label{cor: homotopy of ETHHMUR}
    There is an isomorphism 
    \[ 
        \underline{\pi}_\bigstar \Res_{C_2}^{C_2 \times S^1} \ETHH(MU_\R) \cong \Lambda_{(\underline{MU_\R})_\bigstar} (\overline{\lambda}_n | n \geq 1)
    \]
    of $RO(C_2)$-graded Green functors, where $\overline{\lambda}_n$ has degree $n \rho+1$ and $\rho$ denotes the regular representation.
\end{corollary}

\begin{proof}
    By \cref{cor-ETHH-MUR}, this is the same thing as computing the homology graded Green functor
    \[ 
        (\underline{MU_\R})_\bigstar(BBU_\R) . 
    \]
    To start with, the Thom isomorphism  implies
    \[ 
        (\underline{MU_\R})_\bigstar(BU_\R) \cong (\underline{MU_\R})_\bigstar(MU_\R) \cong (\underline{MU_\R})_\bigstar[\overline{b}_n | n \geq 1]
    \]
    where $\overline{b}_n$ has degree $n \rho$.  This last step is due to Araki \cite{Araki}; see also the proof of Theorem 4.28 of \cite{Hill:Freeness}. It follows that $(\underline{MU_\R})_\bigstar(BU_\R)$ is free as a Green functor over $(\underline{MU_\R})_\bigstar$, hence flat by \cite[Theorem A]{HMQ}, so that \cref{thm-ebar} applies. Consequently there is a spectral sequence 
    \[
        \mathrm{Tor}_*^{\underline{MU_\R}_\bigstar(BU_\R)}(\underline{MU_\R}_\bigstar, \underline{MU_\R}_\bigstar) \Rightarrow \underline{MU_\R}_{*+\bigstar}(BBU_\R)
    \]
    of $RO(C_2)$-graded Green functors.

    Using the Koszul complex as in \cite[proof of Proposition 4.3.3]{AGHKK22}, we compute 
    \[
        \mathrm{Tor}_*^{\underline{MU_\R}_\bigstar[\overline{b}_n | n \geq 1]}(\underline{MU_\R}_\bigstar, \underline{MU_\R}_\bigstar) \cong \Lambda_{(\underline{MU_\R})_\bigstar}(\overline{\ell}_n | n \geq 1) 
    \]
    where $\overline{\ell}_n$ has $(\Z,RO(C_2))$-bidegree $(1,n \rho)$. The differential $d_r$ of the spectral sequence sends an element of $(\Z,RO(C_2))$-bidegree $(s,t)$ to an element of bidegree $(a-r,t+r-1)$. Since the spectral sequence is zero whenever the first bidegree is negative, each $\overline{\ell}_n$ is a permanent cycle. Since the spectral sequence is multiplicative, it collapses at the $E^2$-page. 
    
    Finally, since the spectral sequence is one of $\underline{MU_\R}_\bigstar$-modules and the $E^\infty$ page is free over $\underline{MU_\R}_\bigstar$, there are no nontrivial additive extensions, so we obtain an additive isomorphism 
    \[ 
        \Lambda_{(\underline{MU_\R})_\bigstar} (\overline{\lambda}_n | n \geq 1) \cong (\underline{MU_\R})_\bigstar (BBU_\R)
    \]
    in which $\overline{\lambda}_n$ lifts $\overline{\ell}_n$. Since the left-hand side is a free polynomial $(\underline{MU_\R})_\bigstar$-algebra on exterior generators $\overline{\lambda}_n$, there are no hidden multiplicative extensions, so the $E^\infty$ page gives the computation.
\end{proof}

\subsection{\texorpdfstring{The computation of $\ETHH(MU_G)$}{The computation of ETHH(MU\_G)}}\label{section: MUG}

In this subsection we compute the homotopy ring $\underline{\pi}_*(\Res^{G \times S^1}_G \ETHH(MU_G))$ when $G$ is a finite abelian group. The case of $G$ the trivial group originally appeared in \cite{MS93}, where the authors credit Andy Baker and Larry Smith for their argument. We give a new argument which specializes to an alternate proof of the computation when $G$ is trivial; ideologically it is a consequence of the identification $\THH(MU) \cong MU \wedge SU_+$ of \cite{BCS10} and the cellular structure of \cite{Yok56}.

When $G$ is nontrivial, the computation is a new result. Our argument in this case requires new ideas beyond ($G$-equivariant versions of) those of Baker--Smith and Yokota---namely, the technique of complex representation disk cellular structures utilized in \cite{Wis24}.

Since $MU_G$ is the $G$-Thom spectrum of a $G$-$\mathbb{E}_\infty$-map $\Omega SU_G \to \Pic(\Spg)$ (see Theorem 2.5.41 of \cite{Sch} for the identification $\Omega SU_G \simeq BU_G$), we obtain the following.

\begin{theorem}\label{cor:identification-of-ETHH-of-MU_G}
    Let $G$ be a finite group. There is an equivalence
    \[
        \mathrm{Res}_{G}^{G \times S^1} \mathrm{ETHH}(MU_G) \simeq MU_G \wedge \Sigma^\infty _+ SU_G
    \]
    of $G$-$\mathbb{E}_\infty$ ring spectra.
\end{theorem}

From here, the idea is to compute $\ETHH(MU_G)$ by constructing a nice cellular filtration on $SU_G$. More precisely, we will observe that Yokota's cellular structure is compatible with the $G$-action, and determines a cellular structure on $SU_G$ in which the cells are products of intervals and unit disks in complex $G$-representations. Now $MU_G$ is complex stable: if $V$ is a complex $G$-representation, then $S^V \wedge MU_G$ is equivalent to $S^{\mathrm{dim}_\R(V)} \wedge MU_G$. The argument proceeds by observing then that Yokota's cellular structure is hand-crafted to ease homology computations, and that this remains true in our $G$-equivariant situation.

From here on out, $G$ is a fixed compact abelian Lie group. We make this assumption for two reasons. First, it implies that the Mackey functor homotopy groups $\underline{\pi}_*(MU_G)$ are concentrated in even degrees \cite{Lof73,Com96}. Second, it implies that every complex $G$-representation is a direct sum of complex dimension one $G$-representations. By a complex $G$-representation $V$, we mean a complex inner product space equipped with a group homomorphism $G \rightarrow U(V)$ from $G$ to the unitary group acting on $V$.

\begin{definition}
     Define the $G$-space $SU_G(V)$ as the special unitary group which acts on $V$ equipped with the $G$-action by conjugation in $SU(V) \subset U(V)$. Define $SU_G$ as the colimit of the $SU_G(V)$ over all inclusions of finite-dimensional $G$-representations.
\end{definition}

We emphasize that $SU_G(V)$ is just notation for a particular $G$-action on $SU(\mathrm{dim}_\C(V))$.

Choose a sequence of one-dimensional $G$-representations $V_i$ so that $\oplus_{i=1}^\infty V_i$ is a complete $G$-universe. Equivalently, the sequence $V_1, V_2, \dots$ contains each (isomorphism class of an) irreducible $G$-representation infinitely many times. Writing $W_n = \oplus_{i=1}^n V_i$, it follows that we have an identification
\[
    SU_G \cong \colim_n SU_G(W_n) . 
\]

Yokota constructs a cellular structure on non-equivariant $SU(n)$ \cite{Yok56}. We recall the essential details here. To start with, Yokota defines an injective map 
\[ 
    f_n : \Sigma^1 \C P(\C^n) \rightarrow SU(n) 
\] 
as follows. Letting $\theta$ denote the suspension coordinate and 
\[ 
    x = [x_1 : \cdots : x_n] \in \C P(\C^n) 
\]
with $(x_1,\dots,x_n)$ chosen to be a point on the unit sphere in $\C^n$, define
\[ 
    f_n(\theta,x) := VW 
\] 
where 
\[
    W= \begin{bmatrix} 
    \exp(-2i\theta) & 0 \\
    0 & \mathrm{Id}_{n-1}
    \end{bmatrix} 
\] 
and 
\[ 
    V = 
    \left[
    \mathrm{Id}_n - 2 \exp(-i \theta) \cos \theta 
    \begin{pmatrix}
        x_1 \\ 
        x_2 \\
        \vdots \\
        x_n
    \end{pmatrix}
    \begin{pmatrix}
        \overline{x_1} & \overline{x_2} & \dots & \overline{x_n}
    \end{pmatrix}
    \right] .
\]
Using the matrix determinant lemma 
\[ 
    \mathrm{det}(A+\mathbf{u}\mathbf{v}^T) = \mathrm{det}(A)+\mathbf{v}^T \mathrm{adj}(A)\mathbf{u} ,
\] 
where $\mathrm{adj}$ denotes the adjugate matrix of $A$, one may check that $f_n(\theta,x)$ is indeed a special unitary matrix.

\begin{remark}
    We recall the CW-structure on complex Grassmannians given by Schubert cells. The Schubert cells of $\mathrm{Gr}_k(V)$ are determined by a choice of complete flag for $V$, that is, a sequence of subspaces $0 = V_0 \subset \dots \subset V_n = V$ with $\mathrm{dim}_{\mathbb{C}}(V_i) = i$. We will only require the case $k = 1$, so $\mathrm{Gr}_k(V) = \C P(V)$. The $i$th Schubert cell of $\C P(V)$ is defined as the set
    \[ 
        \{ \ell \in \C P(V) \mid \ell \subset V_i, \ell \nsubseteq V_{i-1} \} . 
    \]
    In other words, it is the subset $\C P(V_i) - \C P(V_{i-1})$ of $\C P(V)$, and it is a cell of dimension $2(i-1)$.
\end{remark}

\begin{theorem}[Yokota]\label{thm:Yokota}
    $SU(n)$ has a finite CW-complex structure with the following properties:
    \begin{enumerate}
        \item the inclusion $SU(n-1) \rightarrow SU(n)$ is cellular and preserves the notation used for each cell,
        \item $f_n$ is cellular and the image of the top cell of $\Sigma^1 \C P(\C^n)$ is the ``primitive" cell $e_n$ of dimension $2n-1$,
        \item the remaining cells $e_{k_1,k_2,\dots,k_q}$ are given by the composition
        \[ 
            e_{k_1} \times \dots \times e_{k_q} \rightarrow SU(n) \times \dots \times SU(n) \xrightarrow{\mathrm{mult.}} SU(n) 
        \]
        in which the first map is the product of the primitive cell inclusion maps; $n\geq k_1>k_2>\dots k_q\geq 2$ (and there is also a zero cell $e_0$).
    \end{enumerate}
\end{theorem}

From the definition of $f_n$, we see that Yokota's map can be made equivariant.

\begin{corollary}\label{cor:Yokota's-map-is-equivariant}
    Yokota's map $\Sigma^1 \C P(\C^n) \rightarrow SU(n)$ refines to a $G$-equivariant map
    \[ 
        \Sigma^1 \C P(W_n) \rightarrow SU(W_n) . 
    \]
\end{corollary}

With this, we may obtain explicit computations.

\begin{theorem}\label{thm:MU_G-homology-of-SU_G-for-G-abelian}
    Let $G$ be compact abelian Lie group and $H \subset G$. Then we have
    \[
        \pi^H_*(MU_G \wedge (SU_G)_+) \cong (MU_H)_* \otimes_\Z \Lambda(\lambda_1,\lambda_2,...)
    \]
    where $\lambda_i$ is an exterior algebra generator in degree $2i+1$.
\end{theorem}

\begin{proof}
    First, write $SU_G$ as the sequential colimit of the $SU_G(W_n)$. It suffices to show that $(MU_G)_*(SU_G(W_n))$ is the algebra $(MU_G)_* \otimes_\Z \Lambda(\lambda_1,\dots,\lambda_{n-1})$, and we shall induct on $n$. For the base case $n=1$, suppose we have picked $W_1=V_1$ to be the trivial $G$-representation on $\mathbb{C}$.  Then the formula follows from the isomorphism $SU_G(W_1)\cong SU_G(\mathbb{C})\cong S^1$, where the circle has trivial action.

    Note that Yokota uses the standard Schubert cells on complex projective space. In the case of $\C P(W_n)$, \cite[Lemma 5.2]{Wis24} 
    says that these cells are $G$-equivariantly homeomorphic to unit disks in certain complex representations.

    Write $e := e_{n,k_2,\dots,k_q}$. By \cref{thm:Yokota,cor:Yokota's-map-is-equivariant}, we have a cellular, $G$-equivariant map 
    \[
        \Phi \colon  \Sigma^1 \C P(W_n) \times \Sigma^1 \C P(W_{k_2}) \times \dots \times \Sigma^1 \C P(W_{k_q}) \rightarrow SU(W_n)
    \]
    which sends the top cell of the domain to the cell $e$. This top cell is $G$-equivariantly homeomorphic to the unit disk in $V \times \R^s$ for some complex $G$-representation $V$ and $s \geq 0$.
    
    By a computation of Cole (found more easily in the literature as \cite[Theorem 4.3]{CGK00}), each $\Sigma^1 \C P(W_n)$ has free finitely generated $MU_G$-homology of rank $n-1$, hence the equivariant K\"{u}nneth theorem \cite{LM} implies the $MU_G$-homology of the domain of $\Phi$ is free and finitely generated. Crucially, this implies that the attaching map of the top cell of the domain of $\Phi$ induces zero in $MU_G$-homology. The pushforward of this attaching map along $\Phi$ yields the attaching map for $e$; it thus induces zero in $MU_G$-homology. 
    
    We have shown $(MU_G)_*(SU(W_n))$ has the desired additive form: the cell $e$ determines the class of the product $\lambda_{n-1} \lambda_{k_2-1} \dots \lambda_{k_q-1}$. The multiplicative structure is determined as follows: $\lambda_i$ is in an odd degree and $(MU_G)_*(SU(W_n))$ is free over $\Z$, so it is exterior. Next, the $\lambda_i$ multiply together as claimed using the ring structure coming from the $\mathbb{A}_\infty$ multiplication on $SU_G$ because each general cell $e$ was defined as the external product of the ``primitive" cells $e_k$. This ring structure is the one we started with by an Eckmann--Hilton argument.
\end{proof}

From \cref{cor:identification-of-ETHH-of-MU_G}, which identifies $\ETHH(MU_G)$ when $G$ is finite, we immediately obtain the following.

\begin{corollary}\label{cor:ETHH-of-MU_G-for-G-finite-abelian}
    Let $G$ be finite abelian group and $H$ any subgroup. Then we have
    \[
        \pi^H_*(\ETHH(MU_G)) \cong MU_H^* \otimes_\Z \Lambda(\lambda_1,\lambda_2,...)
    \]
    where $|\lambda_i| = 2i+1$.
\end{corollary}

\section{\texorpdfstring{Computing $\ETHH(\HFp)$ for $p$ odd}{Computing ETHH(HF\_p) for p odd}}\label{sec: Bokstedt}

In this section, we compute the homotopy groups of $\ethh(\HFp)$ for $p$ odd, where $\HFp$ denotes the equivariant Eilenberg--MacLane spectrum of the constant $C_p$-Mackey functor $\underline{\mathbb{F}}_p$. We remark that for $p=2$ and $C_2$, the answer is already known additively. By Theorem 5.2.2 of \cite{AGHKK22}, there is an equivalence of genuine $C_2$-spectra,
$$\ETHH (H\underline{\mathbb{F}}_2) \simeq  H\underline{\mathbb{F}}_2 \wedge \Omega^\sigma S^{\rho +1}_+ \simeq H\underline{\mathbb{F}}_2 \wedge (\bigvee_{k \geq 0} S^{2k\rho} \vee \bigvee_{k \geq 0} S^{2k\rho +2}).$$
The multiplicative structure on the homotopy groups will be computed in forthcoming work of Chan--Vogeli.

\subsection{\texorpdfstring{The $C_p$-geometric fixed points of $\ETHH(\HFp)$}{The C\_p-geometric fixed points of ETHH(HF\_p)}}

Let $G = C_p$, and let $X$ be a genuine $G$-spectrum.  Let $EG$ be a free contractible $C_p$-space, and let $\widetilde{E}G$ denote the reduced suspension of $EG$.  We write
\begin{align*} \label{diag:TateSquare}
    X^h & = F(EG_+,X) \\
    X^\Phi & = \widetilde{E}G\wedge X\\
    X^t & = F(EG_+,\widetilde{X}).
\end{align*}
Note that the geometric fixed points of $X$ are $X^{\Phi C_p}=(X^{\Phi})^{C_p}$.  The Tate square is homotopy pullback square of genuine $G$-spectra
\[
    \begin{tikzcd}
        X \ar[r] \ar[d] & X^{\Phi} \ar[d]\\
        X^h\ar[r] & X^t
    \end{tikzcd}
\]
where the vertical maps are defined by pullback along the collapse $EG_+\to S^0$ and the horizontal maps are obtained by smashing $X$ and $X^h$ with the canonical map $S^0\to \widetilde{E}G$.

Since this square is a homotopy pullback the homotopy fibers of the horizontal maps are equivalent, and are given by $EG_+\wedge X$.  Notice that whenever $X$ is connective we have $EG_+\wedge X$ is also connective.  Looking at the induced map of long exact sequences for the long exact sequences induced by the Tate square gives the following lemma.

\begin{lemma}[{\cite[Lemma 5.1.3]{ChanVogeli}}]\label{lem:Greenlees--Meier}
    The map $X\to X^h$ is a connective cover if and only if $X^{\Phi}\to X^{t}$ is a connected cover.
\end{lemma}

\begin{proposition}\label{prop:xyzw}
    The maps $\HFp^{\Phi}\to \HFp^t$ and $H\underline{\mathbb{Z}}^{\Phi}\to H\underline{\mathbb{Z}}^{tC_p}$ are connective covers of $G$-spectra. 
\end{proposition}
\begin{proof}
    By the lemma it suffices to prove that the map $\HFp\to \HFp^h$ is a connective cover.  Since $EG_+$ is non-equivariantly equivalent to a sphere we have that the map $\HFp\to \HFp^h$ is an underlying equivalence and thus it remains to check that this map is an connective cover on $C_p$-fixed points.

    By definition, $\HFp^{C_p} = H\mathbb{F}_p$.  The homotopy groups of the spectrum $\HFp^{hC_p}$ are computed using the homotopy fixed point spectral sequence which collapses immediately and tells us that $\HFp^{hC_p}$ is co-connective and $\pi_0\HFp^{hC_p}\cong \mathbb{F}_p$ It follows that the connective cover of $\HFp^{hC_p}$ is $H\F_p$, and all that remains is to check that the map $\HFp^{C_p}\to \HFp^{hC_p}$ is an isomorphism on $\pi_0$.  But this is a map of ring spectra, hence induces the unique ring endomorphism of $\F_p$ on $\pi_0$ which is an isomorphism.  The proof for $H\underline{\mathbb{Z}}$ is essentially the same.
\end{proof}

\begin{corollary}\label{cor:homotopyGeoFix}
    The integer graded homotopy groups of $\HFp\geop$ are given by the non-negative Tate-cohomology ring $\widehat{H}^{*}(C_p;\F_p)$.  Explicitly, for $*\geq 0$, these are given by
    \[
        \pi_*(\HFp\geop)\cong \widehat{H}^{*}(C_p;\F_p)\cong
        \begin{cases}
            \F_2[x],\ |x|=1 & p=2\\
            \Lambda(x)\otimes \F_p[y],\ |x|=1,\ |y|=2 & p>2.
        \end{cases}
    \]
    Similarly, the homotopy groups of $H\underline{\mathbb{Z}}\geop$  given by the non-negative Tate cohomology ring $\widehat{H}^{*}(C_p;\mathbb{Z})$.  Explicitly, for $*\geq 0$, these are given by 
    \[
        \pi_*(H\mathbb{Z}\geop)\cong \widehat{H}^{*}(C_p;\mathbb{Z})\cong \mathbb{F}_p[y]
    \]
    where $|y|=2$.
\end{corollary}
\begin{remark}\label{remark: induced map on geofix}
    For odd primes, the map $H\mathbb{Z}\geop\to \HFp\geop$ induces the inclusion $\mathbb{F}_p[y]\to \Lambda(x)\otimes \mathbb{F}_p[y]$ on homotopy groups.  
\end{remark}

Moreover, the spectral sequence degenerating at the $E_2$ page is merely a symptom of the Schwede--Shipley theorem, saying that we can think of $\HFp$ as a differential graded algebra rather than a spectrum, with appropriate equivariance.
\begin{corollary}\label{dgaGeoFix}
  As a CDGA, $(\HFp)\geop \we \tau_{\geq 0}\widehat{C}^*(C_p; \F_p)$, where $\widehat{C}$ refers to Tate cochains. Similarly, we have $(H\underline{\mathbb{Z}})\geop \we \tau_{\geq 0}\widehat{C}^*(C_p; \mathbb{Z})$
\end{corollary}

\begin{remark}\label{rem:CDGAmodels}
    Using the standard resolution for $\mathbb{F}_p$ as an $\mathbb{F}_p[C_p]$-module, one computes the differentials in the Tate cochain complex for $\widehat{C}^*(C_p; \F_p)$ as all zero.  In particular, $(\HFp)\geop$ is equivalent to the CDGA given by its graded homotopy ring $\mathbb{F}_p[x,y]/(x^2)$ with differential zero.  Similarly, the Tate cochain complex for $\widehat{C}^*(C_p; \mathbb{Z})$ is a CDGA which is quasi-isomorphic to its homology, and $(H\underline{\mathbb{Z}})\geop$ is equivalent to $\mathbb{F}_p[z]$ with $|z|=2$ and differential zero. Note that because $(H\underline{\mathbb{Z}})\geop$ has a CDGA model which is restricted from the derived category of $\mathbb{F}_p$, we have that $(H\underline{\mathbb{Z}})\geop$ is equivalent to a commutative $H\mathbb{F}_p$-algebra.
\end{remark}

In other words, our ring structure comes from the ring structure of the homotopy fixed points.

\begin{corollary}\label{cor-geomfp-ETHH}
  Since geometric fixed points are a symmetric monoidal left adjoint, it follows that
  \[
    \ethh(\HFp)\geop \we \THH(\HFp\geop) \we \THH(C^*(BC_p;\F_p))
  \]
  as commutative $H\mathbb{F}_p$-algebras.  Similarly, 
  \[ 
     \ethh(H\underline{\mathbb{Z}})\geop \we \THH(H\underline{\mathbb{Z}}\geop) \we \THH(C^*(BC_p;\mathbb{Z}))
  \]
  as commutative $H\mathbb{F}_p$-algebras. 
\end{corollary}

\begin{definition}
  We say an $R$-DGA $A$ is intrinsically formal if any other $R$-DGA with the same homology is quasi-isomorphic to $A$ as an $R$-DGA.
\end{definition}

\begin{example}\label{example: formal DGA}
    An example of an intrinsically formal $\mathbb{E}_2$-DGA is given by 
    \[
        H\mathbb{F}_p\wedge (\Omega S^3)_+\simeq \mathbb{F}_{p}[y],\quad |y|=2,
    \]
    where $\Omega S^3$ is an $\mathbb{E}_2$-space, as it is the loops of the topological group $S^3\cong SU(2)$. A proof that this $\mathbb{F}_p$-DGA is intrinsically formal can be found in \cite[Theorem 3.5]{Hor25} or \cite[Theorem 2.1]{BT22}.
\end{example}
\begin{remark}\label{remark: geo HZ as E2}
    Observe that the homotopy groups of the $H\mathbb{F}_2$-algebra in the example agree with those of $H\underline{\mathbb{Z}}\geop$. It follows from intrinsic formality that there is an equivalence of $\mathbb{E}_2$-algebras in $H\mathbb{F}_p$-modules $H\underline{\mathbb{Z}}\geop \simeq H\mathbb{F}_p\wedge \Omega S^3_+$.
\end{remark}

\begin{proposition}
  \label{E2AlgebraGeoFix}
  For $p$ odd, there is an equivalence of $H\F_p$-$\mathbb{E}_2$-algebras
  \[ \HFp\geop \we H\F_p \smsh S^1_+ \smsh (\Omega S^3)_+. \]
\end{proposition}
\begin{proof}
 Note that $E$ is a commutative ring spectrum; we know the homotopy groups are $\Lambda(x)\otimes \F_p[y]$ with $|x|=1$ and $|y|=2$ by \cref{cor:homotopyGeoFix}. 
 There is an $\mathbb{E}_\infty$-ring map 
 \[ 
    \hat x:H\F_p \smsh S^1_+ \to E 
 \] 
 picking out the generator $x$ since $H\F_p\smsh S^1_+\we \F_p[x]/x^2$ with $|x|=1$ is intrinsically formal as a CDGA, which can be straightforwardly checked from the definition. Similarly 
 \[ 
    H\F_p \smsh (\Omega S^3)_+ \we H\F_p\smsh (\Omega\Sigma S^2)_+ \we \F_p[y] 
 \] 
 is intrinsically formal as an $\mathbb{E}_2$-DGA with $|y|=2$ (\cite[Theorem 3.5]{Hor25} or \cite[Theorem 2.1]{BT22}), hence we have an $H \F _p$-$\mathbb{E}_2$ map $H\F_p \smsh (\Omega S^3)_+ \to E$ picking out $y$.

 Together this gives an $\mathbb{E}_2$-$H\F_p$-algebra map 
 \[ 
    (H\F_p\smsh S^1_+) \smsh_{H\F_p} (H\F_p\smsh (\Omega S^3)_+) \to E 
 \] 
 which is surjective on $\pi_*$, and hence an equivalence since each individual graded part of $\pi_*$ is a finite dimensional vector space.
\end{proof}

\begin{theorem}\label{prop-THH-geomfp}
    For $p$ odd, there is an equivalence of $H\mathbb{F}_p$-$\mathbb{E}_1$-algebras
    \[
    \ethh(\HFp)\geop \simeq H\F_p \wedge (\Omega S^3)_+ \wedge S^1_+ \wedge \C P^\infty _+ \wedge (LS^3)_+ 
    \]
    where $LS^3$ denotes the free loop space of $S^3$.
\end{theorem}
 
\begin{proof}
    By Corollary \ref{cor-geomfp-ETHH}, we have an $\mathbb{E}_\infty$-identification
 $$\ethh(\HFp)\geop \simeq \THH(\HFp\geop)$$
 which, by Proposition \ref{E2AlgebraGeoFix}, is equivalent, as an $\mathbb{E}_1$-$H\F_p$-algebra, to 
 $$\THH ( H \F_p \wedge S^1_+ \wedge (\Omega S^3)_+).$$
 Combining the fact that $\THH(H \F _p) \simeq H\F _p \wedge (\Omega S^3)_+$ as $\mathbb{E}_2$-algebras \cite[Theorem 1.3]{BT22} with the THH of suspension spectra, we get
 $$\ETHH(\HFp)\geop \simeq H\F _p \wedge (\Omega S^3)_+ \wedge (LBS^1)_+ \wedge (LS^3)_+$$
 and since $S^1$ is commutative, this is equivalent to
 $$H\F _p \wedge (\Omega S^3)_+ \wedge S^1_+ \wedge \C P^\infty _+ \wedge (LS^3)_+$$
 as $H\F _p$-algebras, as required.
\end{proof}

\begin{remark}
  Because $S^3\cong SU(2)$ is an $\mathbb{E}_1$-space, there is a homotopy equivalence of spaces
  \[ (LS^3)_+ \we (\Omega S^3)_+ \smsh S^3_+, \]
  and so of suspension spectra, but it does not promote to an equivalence of $\mathbb{E}_1$-spectra.
\end{remark}

We will require some $\F_p$-homology calculations.

\begin{lemma}
    Let $p$ be odd.
    \begin{alignat*}{3}
        H_*(\Omega S^3;\F_p) & \cong \F_p[b], &\quad |b| &= 2 \\
        H_*(S^1;\F_p) & \cong \Lambda_{\F_p}[d], &\quad |d| &= 1 \\
        H_*(\C P^\infty;\F_p) & \cong \Gamma_{\F_p}[a], &\quad |a| &= 2 \\
        H_*(LS^3;\F_p) & \cong \F_p[c] \otimes \Lambda_{\F_p}[e], &\quad |c| &= 2, |e| = 3
    \end{alignat*}
    where $\Gamma_{\F_p}$ denotes a divided power algebra of $\F_p$.
\end{lemma}
\begin{proof}
    The first three are standard (cf. \cite[Section 3.C]{Hatcher} for the Pontryagin ring). For the last isomorphism, we consider the Serre spectral sequence of the fibration of $\mathbb{E}_1$-spaces $\Omega S^3\to LS^3\to S^3$.  This is a spectral sequence of algebras with $E^2$-page $\mathbb{F}_p[c]\otimes\Lambda_{\mathbb{F}_p}[e]$ with $|d| = (0,2)$ and $|e| = (3,0)$.  For degree reasons we have $E^3=E^{\infty}$, and by the Leibniz rule all possible differentials are determined by is $d_2(e)$, and we observe that this must be zero.  Indeed, we must have $d_2(e) = kd$ for some $k\in \mathbb{F}_p\setminus 0$, and if $k\neq 0$ then the $E^3=E^{\infty}$ page tells us that $H_2(LS^2;\mathbb{F}_p)=0$.  On the other hand, the homotopy equivalence of spaces $LS^3\simeq S^3\times \Omega S^3$ implies that that $H_2(LS^2;\mathbb{F}_p)\cong \mathbb{F}_p$.

    Thus, the $E^{\infty}$-page of the Serre spectral sequence is isomorphic to $\mathbb{F}_p[c]\otimes\Lambda_{\mathbb{F}_p}[e]$.  Since this is a free graded commutative algebra, and this is a spectral sequence of algebras, this is isomorphic to $H_*(LS^3;\mathbb{F}_p)$; see \cite[Example 1.K]{McC} for the last step.
\end{proof}
Combining the last two results and K\"unneth isomorphism yield the following corollary.
\begin{corollary}\label{cor: homotopy of geo of ETHH hfp}
    There is an isomorphism of graded rings
    \[
        \pi_*(\ETHH(\HFp)\geop)\cong \Gamma[a]\otimes\F_p[b,c]\otimes \Lambda[d,e]
    \]
    where $|a|=|b|=|c|=2$, $|d|=1$, and $|e|=3$.
\end{corollary}

\subsection{The $C_p$-homotopy fixed points of $\ETHH(\HFp)$.}
Let $X$ be a genuine ring $G$-spectrum. On underlying spectra with $G$-action, $\ethh(X)$ is $\THH(X^e)$, where $\THH$ is computed as the geometric realization of the cyclic bar construction in the category $\Fun(BG,\Sp)$ (assuming that $X$ is cofibrant as a ring spectrum).

Since colimits and smash products in $\Fun(BG,\Sp)$ are computed pointwise, we can compute $\THH$ in the category of spectra. Moreover, taking fixed points, in this case $e$-fixed points, is lax symmetric monoidal and so preserves algebra structure.

Since $\HFp$ has trivial $C_p$-action, we have an identifications of $\mathbb{E}_2$-rings
\[
\ethh(\HFp)^h\we \ethh(\HFp)^e \we \THH(H\F_p)
\] with trivial $C_p$-action. We can compute $\ethh(\HFp)\hfixp$ using the homotopy fixed point spectral sequence (HFPSS).

\begin{proposition}
    \label{thhHFpHomotopyFixedPointGroups}
    The HFPSS for computing $\pi_* \THH(H\F_p)^{hC_p}$ collapses at the $E^2$-page.  Moreover, we have an isomorphism
    \[
        \pi_* \THH(H\F_p)^{hC_p}\cong \Lambda(x) \otimes \F_p[y,b]
    \]
    where $|x|=-1$, $|y|=-2$, and $|b|=2$.
\end{proposition}
\begin{proof}
    The HFPSS has $E^2$-page given by
    \[
        E^2_{s,t} = H^{-s}(C_p;\pi_t(\THH(H\F_p)))\cong \Lambda(x) \otimes \F_p[y,b]
    \] 
    with bidegrees $|x| = (-1,0)$, $|y| = (-2,0)$, and $|b| = (0,2)$.  Since this is a free bigraded commutative algebra, we will be done as soon we show that the spectral sequence collapses at the $E^{\infty}$-page.  

    As this is a spectral sequence of algebras, it suffices to prove that $x$, $y$, and $b$ are permanent cycles.  For the elements $x$ and $y$, we observe that because $\THH(H\F_p)$ is an augmented $H\F_p$-algebra we have that the HFPSS for $H\F_p^{hC_p}$ is a retract of the spectral sequence for $ \THH(H\F_p)^{hC_p}$ by functoriality of the spectral sequence.  The former spectral sequence has $E^2$-page given by $\Lambda(x) \otimes \F_p[y]$ which collapses immediately for degree reasons and hence $x$ and $y$ are permanent cycles.

    To see that $b$ is a permanent cycle, we use the fact that $\THH(H\F_p)$ has trivial action to write the homotopy fixed points as
    \[
        \THH(H\F_p)^{hC_p}\simeq \Fun((BC_p)_+,\THH(H\F_p)).
    \]
    The HFPSS comes from the skeletal filtration on $(BC_p)_+$, and the subring $E^2_{0,t}$ comes from the inclusion of the zero skeleton $\mathbb{S}\to (BC_p)_+$.  But this is the inclusion of a retract, and hence the $y$ axis consists entirely of permanent cycles; in particular $b\in E^2_{0,2}$ is a permanent cycle.
\end{proof}

\subsection{The $C_p$-Tate construction of $\ETHH(\HFp)$.}

We now turn to studying the Tate fixed points of $\ethh(\HFp)$.  Similar to the homotopy fixed points we have an equivalence of $\mathbb{E}_{\infty}$-rings
\[
    \ethh(\HFp)^{tC_p}\simeq \THH(H\F_p)^{tC_p}
\]
and we can compute the homotopy groups using the Tate spectral sequence.

\begin{lemma}
  \label{hfixToTatepIsInjective}
  The canonical map 
  \[f\colon \ethh(\HFp)\hfixp \to \ethh(\HFp)\tatep\]
  induces an injection between the $E_2$-page of the homotopy fixed point spectral sequence and that of the Tate spectral sequence.
\end{lemma}
\begin{proof}
  Since $\HFp$ has trivial action, we can check that the norm map from orbits to fixed points is $p=0$. Therefore $\pi_*(\HFp\tatep) \we \pi_*(\HFp\hfixp)$ for $*\leq 0$. This immediately gives the inclusion of spectral sequences on the $E_2$ page.
\end{proof}

\begin{proposition}
  The Tate spectral sequence for $\ethh(\HFp)\tatep$ collapses on the $E_2$-page, giving
  \[ E_\infty \cong \Lambda(x) \otimes \F_p[x,y,b]. \]
\end{proposition}
\begin{proof}
    The same proof as in the homotopy fixed point case (\cref{thhHFpHomotopyFixedPointGroups}) works here.
\end{proof}

\begin{proposition}
  \label{thhTateHomotopy}
  For $p$ odd, there is an isomorphism of graded rings
  \[\pi_*(\ethh(\HFp)\tatep) \we \Lambda(x) \otimes \mathbb{F}_p[y^{\pm 1},b]\] with $|y|=-2$, $|x|=-1$, and $|b|=2$.
\end{proposition}
\begin{proof}
  Since the canonical map $f$ from homotopy fixed points to Tate of $\ethh(\HFp)$ induces an injection of $E_2$-pages and the spectral sequences collapse immediately, we get an injection of $E_\infty$-pages, and so on the actual homotopy groups as well.  Moreover, the inclusion of $E_{\infty}$-pages is an isomorphism in all negative total degrees, and hence the map on homotopy groups is an isomorphism in all negative degrees. 
  
   As in the case of homotopy fixed points, since $\THH(\HFp\tatep)$ is an augmented $\HFp\tatep$ algebra, $\ethh(\HFp)\tatep$ has $\HFp\tatep$ as a retract, which in particular means that $y$ is inverted in the homotopy groups of $\ethh(\HFp)\tatep$.

  Observe that we have a commuting triangle
  \[
  \begin{tikzcd}
      \Lambda(x) \otimes \F_p[y,b]\ar[r,"f"] \ar[d,"\ell"] 
      & \pi_*(\ethh(\HFp)\tatep)
      \\
      \Lambda(x) \otimes \F_p[y^\pm,b] \ar[ur,"g"']
  \end{tikzcd}
  \]
  where $f$ is the map from the free graded-commutative algebra, $\ell$ is the localization, and $g$ is the map induced by the universal property of localization, as $f(y)$ is a unit.  Since the maps $\ell$ and $f$ are both isomorphisms in negative degrees, we see that $g$ is also an isomorphism in negative degrees.  Finally, we claim that $g$ is an isomorphism in all degrees.  Indeed, consider the commutative square
  \[
    \begin{tikzcd}
         \left(\Lambda(x) \otimes \F_p[y^\pm,b]\right)_n \ar[r,"g_n"] \ar[d,"y^k"] 
         & \pi_n(\ethh(\HFp)\tatep) \ar[d,"g(y)^k"]
         \\
        \left(\Lambda(x)\otimes \F_p[y^\pm,b]\right)_{n-2k} \ar[r,"g_{n-2k}"] 
        & \pi_{n-2k}(\ethh(\HFp)\tatep) 
    \end{tikzcd}
  \]
  which commutes because $g$ is a map of rings.  The vertical maps are isomorphisms because $y$ and $g(y)$ are units.  For any $k$ such that $n-2k<0$ the bottom horizontal map is an isomorphism, hence the top horizontal map is an isomorphism for all $n$.
\end{proof}

\subsection{Final answer}

We now compute $\ethh(\HFp)$ using the Tate square.  To clean up notation, we will write $T = \ethh(\HFp)$ throughout this section. Recall that we have a homotopy pullback of $\mathbb{E}_{\infty}$-ring spectra
\begin{equation}\label{eq: pullback square}
    \begin{tikzcd}
        T^{C_p}\ar[r,"i"] \ar[d] & T\geop \ar[d,"q"]\\
        T^{hC_p} \ar[r,"j"] & T^{tC_p}
    \end{tikzcd}
\end{equation}
which induces the usual Mayer--Vietoris long exact sequence on homotopy groups. 

For convenience, we recall that homotopy rings of the known pieces:
\begin{itemize}
    \item $\pi_*(T^{hC_p})\cong \Lambda(x) \otimes \mathbb{F}_p[y,b]$, where $|x|=-1$, $|y|=-2$, and $|b|=2$,
    \item $\pi_*(T^{tC_p})\cong \Lambda(x) \otimes \mathbb{F}_p[y^{\pm},b]$, where $|x|=-1$, $|y|=-2$, and $|b|=2$,
    \item $\pi_*(T\geop) \cong \mathbb{F}_{p}[b,c]\otimes \Lambda_{\mathbb{F}_p}[d,e]\otimes \Gamma[a]$, where $|a|=|b|=|c| =2$, $|d|=1$ and $|e|=3$.
\end{itemize}
The labels have been chosen so that elements labeled by the same letter are sent to one another. For instance, the map $j_*$ is given by the inclusion
\[
    \Lambda(x) \otimes \mathbb{F}_{p}[b,y] \to \Lambda(x) \otimes \F_p[b,y^\pm]
\]
which is evidently injective. 
In fact we show that the other horizontal map is also injective on homotopy.

\begin{lemma}\label{lemma: injection from gen to geo}
    Taking homotopy groups of \eqref{eq: pullback square}, the map $i_*\colon \pi_*(T^{C_p})\to \pi_*(T\geop)$ is always injective.
\end{lemma}
\begin{proof}
    Consider the following diagram,
\[\begin{tikzcd}[ampersand replacement=\&]
	{\pi_*(\mathrm{fib}(i))} \& {\pi_*(T^{C_p})} \& {\pi_{*}(T\geop)} \& {\pi_{*-1}(\mathrm{fib}(i))} 
    \\
	{\pi_*(\mathrm{fib}(j))} \& {\pi_*(T^{hC_p})} \& {\pi_*(T^{tC_p})} \& {\pi_{*-1}(\mathrm{fib}(j))}
	\arrow[from=1-1, to=1-2]
	\arrow["\cong", from=1-1, to=2-1]
	\arrow["{i_*}", hook, from=1-2, to=1-3]
	\arrow[from=1-2, to=2-2]
	\arrow["{\partial_1}", two heads, from=1-3, to=1-4]
	\arrow["{q_*}", from=1-3, to=2-3]
	\arrow["\cong", from=1-4, to=2-4]
	\arrow["0", from=2-1, to=2-2]
	\arrow["{j_*}", hook, from=2-2, to=2-3]
	\arrow["{\partial_2}", two heads, from=2-3, to=2-4]
\end{tikzcd}\]
 where the horizontal rows are the long exact sequences associated to the fibrations and the vertical isomorphisms are due to the Adams isomorphism, as is typical of the Tate square.
 The map $j_*$ is injective, as discussed above, which implies that the map $\partial_2$ is surjective.  This, in turn, implies that $\partial_1$ is surjective and therefore $i_*$ is also injective.
\end{proof}
In particular, we can consider $\pi_*(T^{C_p})$ as a subring of $\pi_*(T\geop)$. 

\begin{corollary}\label{cor: subring}
    The subring $\pi_*(T^{C_p})\subset \pi_*(T\geop)$ is equal to the preimage 
    \[q^{-1}(\mathrm{im}(j\colon \pi_*(T^{hC_p})\to \pi_*(T^{tC_p}))).\]
\end{corollary}
\begin{proof}
    The fact that $i_*$ is injective implies that the Mayer--Vietoris long exact sequence associated to \eqref{eq: pullback square} has boundary maps which are all zero.  Consequently, the squares
    \[
        \begin{tikzcd}
        \pi_*(T^{C_p}) \ar[hook,r,"i_*"] \ar[d,] & \pi_*(T\geop) \ar[d,"q_*"]\\
        \pi_*(T^{hC_p}) \ar[hook, r,"j_*"] & \pi_*(T^{tC_p})
        \end{tikzcd}
    \]
    are all pullbacks of sets.  The claim then follows from the fact that the horizontal maps are injections.
\end{proof}

In particular, if we can achieve a good understanding of how the map
\[
    q_*\colon \mathbb{F}_{p}[b,c]\otimes \Lambda_{\mathbb{F}_p}[d,e]\otimes \Gamma[a]
    \to \Lambda(x)\otimes \mathbb{F}_p[y^{\pm},b]
\]
acts on the generators $a,b,c,d,e$, we will be able to reconstruct the homotopy of $T^{C_p}$. As a first step, observe that the naturality of the map from geometric fixed points to the Tate construction gives a commutative diagram of commutative $H\F_p$-algebras,
    \begin{equation}\label{eq: blah}
    \begin{tikzcd}
    \HFp\geop \ar[r] \ar[d] &
    T\geop \ar[d,"q"]
    \\
    \HFp^{t C_p} \ar[r,"\ell"] &
    T^{t C_p}.
    \end{tikzcd}
    \end{equation}
  Here, the horizontal maps are induced by the unit maps making $\THH(R)$ an (augmented) $R$-algebra.

\begin{lemma}
  \label{lemma:imageOfEll}
  Taking homotopy groups of \eqref{eq: pullback square}, the map $q_*$ takes the elements $c$ and $d$ to the image of $\ell_*$.
\end{lemma}
\begin{proof}
  Taking homotopy groups of \eqref{eq: blah} gives a diagram
  \[ \begin{tikzcd}
    \F_p[c]\otimes\Lambda(d) \ar[r,hook] \ar[d] &
    \F_p[b,c] \otimes \Lambda(d,e) \otimes \Gamma[a] \ar[d,"q_*"]
    \\
    \Lambda(x) \otimes \F_p[y^{\pm}] \ar[r,hook,"\ell_*"] &
    \Lambda(x) \otimes \F_p[y^\pm,b].
  \end{tikzcd} \]
  The horizontal maps in \eqref{eq: blah} are inclusions of retracts, hence the horizontal maps here are injections and the result follows.
\end{proof}
 In the following lemma, we write $\gamma_n(a)$ for the $n$-th divided power of $a$.

\begin{proposition}\label{prop: most of q}
    The map $q_*$ acts as follows:
    \begin{align*}
        q_*(\gamma_n(a)) & =0 \\
        q_*(b) &= b \\
        q_*(c) &= y^{-1} \\
        q_*(d) &= xy^{-1} \\
        q_*(e) &= 0
    \end{align*}
    where each $z_i$ is in the image of $j_i$.
\end{proposition}
\begin{proof}
    We take each relation in turn and use $\F_p\{\alpha\}$ to denote the $\F_p$-module generated by $\alpha$.
    \begin{itemize}
        \item The mod $p$ divided power relations tell us that $\gamma_n(a)^p=0$ for all $n$.  Therefore, $q_*(\gamma_n(a))$ is an even degree nilpotent element in $\Lambda(x)\otimes \mathbb{F}_p[y^{\pm},b]$, which must be zero.
        
        \item The diagram \eqref{eq: pullback square} is a diagram of $\THH(H\F_p)$-algebras, so on homotopy groups all maps take $b$ to $b$.
        
        \item  The short exact sequence associated to \eqref{eq: pullback square} in degree 2 implies that the map
        \[
            \mathbb{F}_p\{b,c,a \}\oplus \mathbb{F}_p\{y^nb^{n+1}\}_{n\geq 0}
            \xrightarrow{q_2\oplus j_2} 
            \mathbb{F}_p\{y^nb^{n+1}\}_{n\geq -1}
        \]
        is surjective. Observe that because $q_*(b)$ and $q_*(a)$ are in the image of $j_2$, we have that $q_*(c)$ generates the cokernel of the image of $j_2$; in particular it is not zero. Now, by \cref{lemma:imageOfEll} it must be that $q_*(c)=ky^{-1}$ for some unit $k\in \F_p$. Replacing $y$ with $k^{-1}y$ if necessary, we have $q_*(c)=y^{-1}$.

        \item Similarly, the short exact sequence associated to \eqref{eq: pullback square} in degree 1 implies that the map
        \[
            \mathbb{F}_p\{d \}\oplus \mathbb{F}_p\{xy^nb^{n+1}\}_{n\geq 0}
            \xrightarrow{q_1\oplus j_1} 
            \mathbb{F}_p\{xy^nb^{n+1}\}_{n\geq -1}
        \]
        is surjective. It follows that $q_*(d)$ generates the cokernel of $j_1$, and hence is non-zero.  Now, by \cref{lemma:imageOfEll} it must be that $q_*(d)=k_2xy^{-1}$ for some unit $k_2\in \F_p$. Replacing $x$ with $k_2^{-1}x$ if necessary, we have $q_*(d)=xy^{-1}$ as desired. Note that it is important that we have modified $x$ here, and not $y$ as we might have modified $y$ to get the previous relation.
    
        \item We defer the computation $q_*(e)=0$ to \cref{cor-eGoesTo0} below.\qedhere
    \end{itemize}
    
\end{proof}
\begin{remark}
    The previous discussion implies that the map $q_1\oplus j_1$ in the long exact Mayer-Vietoris sequence is an isomorphism, so $\pi_1(T^{C_p}) = 0$.
\end{remark}

All that remains is to compute the element $q_*(e)$.  For this, we use a comparison with $S=\ethh(H\underline{\mathbb{Z}})$. Note that the unique ring map $H\underline{\mathbb{Z}}\to \HFp$ induces a map of commutative ring $C_p$-spectra $f\colon S\to T$.  We begin by computing the homotopy groups of the geometric fixed points $S\geop$.

\begin{lemma}\label{lemma: geo of S}
    There is an equivalence of $\mathbb{E}_1$-$H\mathbb{F}_p$-algebras
    \[
        S\geop\simeq H\mathbb{F}_p\wedge \Omega S^3_+\wedge LS^3_+.
    \]
\end{lemma}
\begin{proof}
    By \cref{prop-THH-geomfp}, there is an equivalence of commutative $H\mathbb{F}_p$-algebras $S\geop\simeq \THH(H\underline{\mathbb{Z}}\geop)$. By \cref{remark: geo HZ as E2}, we have an equivalence of $\mathbb{E}_2$-$H\mathbb{F}_p$-algebras $H\underline{\mathbb{Z}}\geop\simeq H\mathbb{F}_p\wedge \Omega S^3_+$, and the claim now follows from the computation of $\THH$ of suspension spectra.
\end{proof}

\begin{corollary}
\label{StoTGeoFixMap}
        The class $e\in \pi_3 T\geop$ is in the image of the map \[f_*\colon S_*\geop\to T_*\geop.\]
\end{corollary}
\begin{proof}
    By \cref{lemma: geo of S,} have equivalences of $\mathbb{E}_1$-$H\mathbb{F}_p$-algebras
    \begin{align*}
        S\geop & \simeq H\mathbb{F}_p\wedge \Omega S^3_+\wedge LS^3_+\\
        T\geop & \simeq H\mathbb{F}_p\wedge \Omega S^3_+\wedge LS^3_+\wedge \mathbb{C}P^{\infty}_+\wedge S^1_+\
    \end{align*}
    and the map $S\geop\to T\geop$ is equivalent to smashing $H\mathbb{F}_p\wedge \Omega S^3_+\wedge LS^3_+$ with the unit map
    \[
        \mathbb{S}\xrightarrow{\eta} \mathbb{C}P^{\infty}_+\wedge S^1_+.
    \]
    This follows from the fact that the map $H\underline{\mathbb{Z}}\geop\to \HFp\geop$ induces the inclusion of the sub-polynomial algebra on the generator in degree $2$. Since the class $e\in \pi_3 T\geop$ comes from the homology of $LS^3_+$, the claim follows.
\end{proof}

\begin{lemma}
\label{lem-StoTTatepDegree3}
    The map $S^{tC_p}\to T^{tC_p}$ induces the zero map on homotopy groups in all odd degrees.
\end{lemma}
\begin{proof}
    Since Tate spectra depend only on underlying spectra, this is equivalent to the same claim about the map $\THH(H\mathbb{Z})^{tC_p}\to \THH(H\mathbb{F}_p)^{tC_p}$, where the $C_p$-actions are trivial. 
    This map is compatible with the Tate spectral sequences
    \begin{align*}
        E^2_{s,t} &= \widehat{H}^{t}(C_p;\THH(H\mathbb{Z}))\Rightarrow \THH(H\mathbb{Z})^{tC_p}\\
        (E^2_{s,t})' &= \widehat{H}^{t}(C_p;\THH(H\mathbb{F}_p))\Rightarrow \THH(H\mathbb{F}_p)^{tC_p}
    \end{align*}
    where the map is induced by the map $\THH(H\mathbb{Z})\to \THH(H\mathbb{F}_p)$.  B\"okstedt computed that $\pi_*(\THH(H\mathbb{Z}))=0$ for $*$ positive and even, but $\pi_*(\THH(H\mathbb{F}_p))=0$ for $*$ odd.  Thus, the map $E^2_{s,t}\to (E^2_{s,t})'$ is zero unless $t=0$.  The claim then follows from the fact that
    \[
        E^2_{s,0} = \widehat{H}^s(C_p;\mathbb{Z})
    \]
    is zero when $s$ is odd.
\end{proof}

\begin{corollary}
\label{cor-eGoesTo0}
    The map $q_*$ takes $e$ to $0$.
\end{corollary}
\begin{proof}
    By naturality, there is a commutative square
    \begin{cd}
        S\geop \ar[r,"f\geop"] \ar[d] & T\geop \ar[d,"q"] \\
        S\tatep \ar[r,"f\tatep"] & T\tatep.
    \end{cd}
    By \cref{lem-StoTTatepDegree3}, $f\tatep_*$ is $0$ in degree $3$, and by \cref{StoTGeoFixMap}, the class $e$ is in the image of $f_*\geop$. The result follows.
\end{proof}

We can now use \cref{prop: most of q,cor: subring} to describe $\pi_*(T^{C_p})$.  Essentially, one must check which monomials in $\mathbb{F}_{p}[b,c]\otimes \Lambda_{\mathbb{F}_p}[d,e]\otimes \Gamma[a]$ are mapped by $q_*$ to elements which are in the image of $j_*$.  This amounts to keeping track of which generators do not have terms with negative powers of $y$.  For instance, the element $c$ is not in $\pi_*(T^{C_p})$, as $q_*(c)=y^{-1}$. However the elements $ec$ and $ac$ are in $\pi_*(T^{C_p})$, as both map to $0$ under $q_*$.

\begin{theorem}\label{thm:Bokstedt}
    The ring $\pi_*(T^{C_p})$ is isomorphic to the subring of $\mathbb{F}_{p}[b,c]\otimes \Lambda_{\mathbb{F}_p}[d,e]\otimes \Gamma[a]$ generated by all monomials of the form $b^n$, or are divisible by either $\gamma_n(a)$ or $e$.
\end{theorem}

\subsection{The case of non-modular characteristic}

We end with a brief subsection explaining how to compute $\ETHH(H\underline{\F})$ in $G$-spectra, when $\F$ is a field of characteristic prime to $|G|$. We write $\underline{\F}$ for the constant $G$-Green functor at the field $\F$. 
\begin{theorem}
    Let $\F$ be a field of characteristic $p$ and let $G$ be a finite group of order coprime to $p$.  There is an isomorphism of graded rings
    \[
        \pi_*^G(\ETHH(H\underline{\F}))\cong \pi_*(\THH(H\F)).
    \]
    In particular, when $\F = \F_p$ we have $\pi_*^G(\ETHH(H\underline{\F}))\cong \F_p[b]$ where $|b|=2$.
\end{theorem}

\begin{proof}
    Since $\Res^{G\times S^1}_{e}\ETHH(H\underline{\F})\simeq \THH(H\underline{\F})$ is an $H\F$-algebra, we see that the homotopy fixed point and Tate spectral sequences computing $\ETHH(H\underline{\F})^{hG}$ and $\ETHH(H\underline{\F})^{tG}$ are both concentrated on the $y$-axis and collapse immediately.  It follows that the maps 
    \[
        \THH(H\F)\to \ETHH(H\underline{\F})^{hG}\to \ETHH(H\underline{\F})^{tG}
    \]
    are both equivalences.  Since the Tate square is a pullback, the map 
    \[\ETHH(H\underline{\F})^G\to \ETHH(H\underline{\F})^{\Phi G}\simeq \THH(H\underline{\F}^{\Phi G})\] 
    is also an equivalence of $\mathbb{E}_{\infty}$-rings.  Finally, applying \cref{lem:Greenlees--Meier} implies that $H\F^{\Phi G}\simeq \tau_{\geq0} H\F^{tG}\simeq H\F$, which proves the claim.
\end{proof}

\bibliographystyle{alpha}
\bibliography{ref}

\end{document}